\documentclass[a4paper,11pt]{article}
\usepackage{mathtools}
\usepackage{amsmath}
\usepackage{mathrsfs}
\usepackage{amssymb}
\usepackage{amsthm}
\usepackage{dsfont}
\usepackage{graphicx}
\usepackage{bmpsize}
\usepackage{times}
\usepackage{ textcomp }
\usepackage{mathtools}
\usepackage{latexsym, empheq, fancybox}
\usepackage{mathrsfs}
\usepackage{exscale}
\usepackage{booktabs, array}
\usepackage[english]{babel}
\usepackage{indentfirst}
\usepackage{enumerate}

\newtheorem{theorem}{Theorem}[section]
\newtheorem{lemma}[theorem]{Lemma}

\newtheorem{definition}[theorem]{Definition}

\usepackage{authblk}
\title{\Large\textbf{  Topological pressure, mistake functions and average metric }
\thanks{Address:School of Mathematical Sciences and Institute of Mathematics, Nanjing Normal University, Nanjing 210023, P.R. China}}
\author{Ercai Chen$^{ \, * }$\thanks{
       {\tt\small  ecchen@njnu.edu.cn}} ,\
Shen He$^{ \, * }$\thanks{{\tt\small\quad hs2398788113@163.com
} }
\ \ and Xiaoyao Zhou$^{\,*}$\thanks{ {\tt\small zhouxiaoyaodeyouxian@126.com}
        }
}

\begin{document}
\date{}

\maketitle
\begin{abstract}
In this paper, we showed that the Pesin pressure of any subset under a mistake function is equal to the classical Pesin pressure of the subset in dynamical systems. Our result extended the result of \cite{chengzhaocao} in additive case, which proved the topological pressure of the whole system is self adaptable  under a mistake function. As an application, we showed that the Pesin pressure defined by average metric is equal to the classical Pesin pressure.

\end{abstract}

\section{Introduction and Preliminaries}
In \cite{ruelle}, Ruelle first introduced the concept of topological pressure for expansive dynamical systems. Later, inspired by the dimension theory,  Pesin and Pitskel \cite{pesinpitskel} extended the notion of topological pressure on any subset. After that, topological pressure has played an important role in thermodynamic formalism. In this paper, we will show the topological pressure of arbitrary subset under a mistake function is equal to the topological pressure of the subset. The motivations of this paper come from two points. One is from \cite{chengzhaocao}. In \cite{chengzhaocao}, Cheng, Zhao and Cao defined the pressure for asymptotically sub-additive potentials under a mistake function. They used the techniques of ergodic theory to show the topological pressure of the whole system under mistake function is the same as the topological pressure without mistake function. A question arises naturally:
Does the topological pressure keep consistent on arbitrary subset? The other motivation comes form \cite{grogerjager}. In \cite{grogerjager}, Gr\"oger and   J\"ager showed that the entropy of the whole system with Bowen metric agrees with the entropy with average metric. It has reason  to ask whether the entropy defined by Bowen metric is the same as the entropy defined by average metric on any subsets. In this paper, we answer the above questions affirmatively. It is mentioning that we only focus on the topological pressure with additive potentials here although Cheng, Zhao and Cao \cite{chengzhaocao} deals with non-additive case.

Throughout the paper, $(X,f)$ is a topological dynamical system (TDS for short) if $(X,d)$ is a compact metric space and $f:X\to X$ is a continuous self-map.

First, we give the definition of the mistake function, which is a little bit different form \cite{pfistersullivan,chengzhaocao}.

 \begin{definition}
 Given $\epsilon_{0}>0$, the function $g:\mathbb{N}\times(0,\epsilon_{0}]\to \mathbb{R}$ is called a mistake function if
 $g(n,\epsilon)\leq g(n+1,\epsilon)$ for all $\epsilon\in(0,\epsilon_0]$ and $n\in\mathbb N$ and
 \begin{align*}
 \lim\limits_{\epsilon\to0}\lim\limits_{n\to\infty}\dfrac{g(n,\epsilon)}{n}=0
 \end{align*}
For a mistake function $g,$ if $\epsilon>\epsilon_0,$ set $g(n,\epsilon)=g(n,\epsilon_0).$
\end{definition}

\begin{definition}
The mistake Bowen ball $B_{n}(g;x,\epsilon)$ centered at $x$ with radius $\epsilon$ and length $n$ associated to the mistake function $g$ is given by the following set:
\begin{align*}
\left\{y\in X:\exists \Lambda\subset\Lambda_{n},\sharp(\Lambda_{n}\setminus \Lambda)\leq g(n,\epsilon) \text{ and } \max\{d(f^{j}x,f^{j}y):j\in\Lambda\} \leq\epsilon\right\}
\end{align*}
where $\Lambda_{n}=\{0,1,2,\cdots,n-1\}$.
\end{definition}
It is obvious that the classical Bowen ball $B_n(x,\epsilon)$ is a subset of  $B_{n}(g;x,\epsilon).$
First, we present the definition of Pesin pressure  via Bowen balls.
\begin{definition}
Let $(X,T)$ be a TDS and $\varphi$ be a continuous function. Given $Z\subset X,\delta>0$ and $N\in\mathbb N,$ let $\mathcal P(Z,N,\delta)$ be the collection of countable sets $\{(x_i,n_i)\}\subset Z\times\{N,N+1,\cdots\}$
such that $Z\subset\bigcup\limits_i B_{n_i}(x_i,\delta).$ For each $s\in\mathbb R,$ consider the set functions
\begin{align*}
m(Z,s,\varphi,N,\delta)&=\inf\limits_{\mathcal P(Z,N,\delta)}\sum\limits_{(x_i,n_i)}\exp(-n_is+S_{n_i}\varphi(x_i)),\\
m(Z,s,\varphi,\delta)&=\lim\limits_{N\to\infty}m(Z,s,\varphi,N,\delta).
\end{align*}
This function is non-increasing in $s,$ and takes value $\infty$ and 0 at all but most one value of $s.$ Denoting the critical value of $s$ by
\begin{align*}
P_Z(\varphi,\delta)=\inf\{s\in\mathbb R:m(Z,s,\varphi,\delta)=0\}=
\sup\{s\in\mathbb R:m(Z,s,\varphi,\delta)=\infty\}.
\end{align*}
The Pesin pressure of $\varphi$ on $Z$ is $P_Z(\varphi)=\lim\limits_{\delta\to0}P_Z(\varphi,\delta).$
\end{definition}
Now, we use the mistake Bowen ball to replace the classical Bowen ball in defining the new topological pressure.
\begin{definition}
Let  $(X,f)$ be a TDS, $\varphi:X\to\mathbb{R}$ be a continuous function $g$ be a mistake function.
Given $Z\subset X$, $\delta>0$, and $N\in \mathbb{N}$.
Let $\mathcal P^g(Z,N,\delta)$ be the collection of countable sets
$\{(x_i,n_i)\}\subset Z\times\{N,N+1,\cdots\}$ such that $Z\subset\bigcup\limits_i B_{n_i}(g;x_i,\delta).$ For each $s\in\mathbb R,$ consider the set functions
\begin{align*}
m^g(Z,s,\varphi,N,\delta)&=\inf\limits_{\mathcal P^g(Z,N,\delta)}\sum\limits_{(x_i,n_i)}\exp(-n_is+S_{n_i}\varphi(x_i)),\\
m^g(Z,s,\varphi,\delta)&=\lim\limits_{N\to\infty}m^g(Z,s,\varphi,N,\delta).
\end{align*}
This function is non-increasing in $s,$ and takes value $\infty$ and 0 at all but most one value of $s.$ Denoting the critical value of $s$ by
\begin{align*}
P_Z(\varphi,\delta,g)=\inf\{s\in\mathbb R:m^g(Z,s,\varphi,\delta)=0\}=
\sup\{s\in\mathbb R:m^g(Z,s,\varphi,\delta)=\infty\}.
\end{align*}
The Pesin pressure of $\varphi$ on $Z$ with mistake function $g$ is $P_Z(\varphi,g)=\lim\limits_{\delta\to0}P_Z(\varphi, \delta,g).$
\end{definition}
In \cite{pesinpitskel}, Pesin and Pitskel used open covers to define the topological pressure as follows:
\begin{definition}
Let $(X,f)$ be a TDS, $\varphi$ be a continuous function. We fix a finite open cover $\mathcal U$ of $X$,
and let $\mathcal S_m(\mathcal U)$ denote the set of all strings ${\bf U}=\{U_{w_1}\cdots U_{w_m}: U_{w_j}\in\mathcal U\}$ of length $m=m({\bf U}).$ Write $\mathcal S(\mathcal U)=\bigcup\limits_{m\geq0}\mathcal S_m(\mathcal U).$ For each ${\bf U}\in\mathcal S(\mathcal U),$ let $X({\bf U})=\{x\in X:f^{j-1}(x)\in U_{w_j} \text{ for all } j=1,\cdots, m({\bf U})\}.$ Given $Z\subset X$ and $N\in\mathbb N,$ we let $\mathcal S(Z,\mathcal U,N)$ denote the set of all finite or countable collections $\mathcal G$ of strings of length at least $N$ which cover $Z.$ We define  set functions by
\begin{align*}
m'(Z,\varphi,\mathcal U,s, N)&=\inf\limits_{\mathcal S(Z,\mathcal U,N)}\left\{\sum\limits_{{\bf U}\in\mathcal G}\exp(-sm({\bf U})+\sup\limits_{x\in Z({\bf U})}S_{m({\bf U})}\varphi(x))\right\}\\
m'(Z,\varphi,\mathcal U,s)&=\lim\limits_{N\to\infty}m'(Z,\varphi,\mathcal U,s, N).
\end{align*}
Denote the critical value of $m'(Z,\varphi,\mathcal U,s)$ by
\begin{align*}
P'_Z(\varphi,\mathcal U)=\inf\{s:m'(Z,\varphi,\mathcal U,s)=0\}=\sup\{s:m'(Z,\varphi,\mathcal U,s)=\infty\}.
\end{align*}
 The Pesin topological pressure of $\varphi$ on $Z$ is given by
 \begin{align*}
 P'_Z(\varphi)=\lim\limits_{{\rm diam } \mathcal U\to0} P_Z'(\varphi,\mathcal U).
 \end{align*}
\end{definition}
Inspired by this, we give a  definition of Pesin pressure with mistake function by using open covers.
\begin{definition}
Let $(X,f)$ be a TDS, $\varphi$ be a continuous function, $g$ be a mistake function. We fix a finite open cover $\mathcal U$ of $X$.
For each ${\bf U}\in\mathcal S_m(\mathcal U),$ let ${\bf U}^g=\{{\bf U}^*:\sharp\{{U_{i_j}^*\neq U_{i_j},j=1,2,\cdots,m\}\leq g(m,{\rm diam~ } \mathcal U) }\},$ and
 let $\mathcal S^g_m(\mathcal U)$ denote the set of such  ${\bf U}^g $ of length $m=m({\bf U}).$
Write $\mathcal S^g(\mathcal U)=\bigcup\limits_{m\geq0}\mathcal S^g_m(\mathcal U).$
For each ${\bf U}^g\in\mathcal S^g(\mathcal U),$ let $X({\bf U}^g)=\{x\in X:\exists {\bf U}^*\in {\bf U}^g \text{ such that } f^{j-1}(x)\in U^*_{w_j} \text{ for all } j=1,\cdots, m({\bf U})\}.$ Given $Z\subset X$ and $N\in\mathbb N,$ we let $\mathcal S^g(Z,\mathcal U,N)$ denote the set of all finite or countable collections $\mathcal G^g$ of strings with mistake function of length at least $N$ which cover $Z.$ We define  set functions by
\begin{align*}
m'(Z,\varphi,\mathcal U,s, N,g)&=\inf\limits_{\mathcal S^g(Z,\mathcal U,N)}\left\{\sum\limits_{{\bf U}^g\in\mathcal G^g}\exp(-sm({\bf U})+\sup\limits_{x\in Z({\bf U}^g)}S_{m({\bf U})}\varphi(x))\right\}\\
m'(Z,\varphi,\mathcal U,s,g)&=\lim\limits_{N\to\infty}m'(Z,\varphi,\mathcal U,s, N,g).
\end{align*}
Denote the critical value of $m'(Z,\varphi,\mathcal U,s,g)$ by
\begin{align*}
P'_Z(\varphi,\mathcal U,g)=\inf\{s:m'(Z,\varphi,\mathcal U,s,g)=0\}=\sup\{s:m'(Z,\varphi,\mathcal U,s,g)=\infty\}.
\end{align*}
 The Pesin topological pressure of $\varphi$ on $Z$ with mistake function $g$ is given by
 \begin{align*}
 P'_Z(\varphi,g)=\lim\limits_{{\rm diam } \mathcal U\to0} P_Z'(\varphi,\mathcal U,g).
 \end{align*}
\end{definition}
In \cite{climenhaga}, Climenhaga showed the definition of classical Pesin pressure by using Bowen ball  is equivalent to the definition defined by using open covers.
Inspired by \cite{climenhaga}, we are going  to show that the two definition of Pesin pressure with mistake function are also equivalent.
\begin{theorem}
Let $(X,f)$ be a TDS, $Z\subset X$, $\varphi$ be a continuous function on $X$, $g$ be a mistake function. Then
\begin{align*}
P_Z'(\varphi,g)=P_Z(\varphi,g).
 \end{align*}
\end{theorem}
\begin{proof}
First we show the following inequality:
\begin{align*}
P_Z'(\varphi,g)\geq P_Z(\varphi,g).
\end{align*}
Given $\delta>0,$ fix an open cover $\mathcal{U}$ of $X$, with diam $ \mathcal{U}<\delta$.
Given $\mathcal G^g\in\mathcal S^g(Z,\mathcal U,N),$ we may assume without loss of generality that for every ${\bf U}^g\in\mathcal G^g$ $Z({\bf U}^g)\cap Z\neq\emptyset.$ Thus for each such ${\bf U}^g,$ we choose $x_{{\bf U}^g}\in Z({\bf U}^g)\cap Z$ with
$Z({\bf U}^g)\subset B(x_{{\bf U}^g},m({\bf U}),\delta)$ and so
\begin{align*}
m'(Z,\varphi,\mathcal U,s, N,g)&=\inf\limits_{\mathcal S^g(Z,\mathcal U,N)}\sum\limits_{{\bf U}^g\in\mathcal G^g}\exp(-sm({\bf U})+\sup
\limits_{x\in Z({\bf U}^g)}S_{m({\bf U})}\varphi(x))\\
&\geq \inf\limits_{\mathcal P^g(Z,N,\delta)}\sum\limits_{(x_i,n_i)}\exp(-n_is+S_{n_i}\varphi(x_i))\\
&=m^g(Z,s,\varphi,N,\delta).
\end{align*}
Thus $P'_Z(\varphi,\mathcal U,g)\geq P_Z(\varphi,\delta,g),$ and taking $\delta\to0,$ we have
\begin{align*}
P'_Z(\varphi,g)\geq P_Z(\varphi,g).
\end{align*}
Next, we show the other inequality of opposite direction.
For $\delta>0,$ let
\begin{align*}
\epsilon(\delta)=\sup\{|\varphi(x)-\varphi(y)|:d(x,y)<\delta\}.
\end{align*}
Since $X$ is compact and $\varphi$ is continuous, we have $\lim\limits_{\delta\to0}\epsilon(\delta)=0.$ Furthermore, given $x\in X, y\in B_n(g;x,\delta),$ we have
\begin{align*}
|S_n\varphi(x)-S_n\varphi(y)|<n\epsilon(\delta)+2g(n,\delta)||\varphi||.
\end{align*}
For a fixed $\delta>0,$ we choose an open cover $\mathcal U$ with diam $(\mathcal U)<\epsilon(\delta).$ Let $L(\mathcal U)$
be the Lebesgue number of $\mathcal U,$ and consider $\{(x_i,n_i)\}\in\mathcal P^g(Z,N,L(\mathcal U))$. Then for each $(x_i,n_i)$ there exists ${\bf U}^g_i\in\mathcal S_{n_i}^g(\mathcal U)$ such that $B_{n_i}(g;x_i,L(\mathcal U))\subset Z({\bf U}_i^g).$
Let $\mathcal G'=\{{\bf U}^g_i\}.$ Then
\begin{align*}
m'(Z,\varphi,\mathcal U,s, N, g)&=\inf\limits_{\mathcal S^g(Z,\mathcal U,N)}\sum\limits_{{\bf U}^g\in\mathcal G^g}\exp(-sm({\bf U})+\sup\limits_{x\in Z({\bf U}^g)}S_{m({\bf U})}\varphi(x))\\
&\leq \sum\limits_{{\bf U}^g_i\in\mathcal G'} \exp(-sm({\bf U}_i)+\sup\limits_{x\in Z({\bf U}_i^g)}S_{m({\bf U}_i)}\varphi(x))\\
&\leq \sum\limits_{(x_i,n_i)}\exp(-n_i(s-\epsilon(\delta))+S_{n_i}\varphi(x_i)+2g(n_i,L(\mathcal U))||\varphi||)\\
&\leq \sum\limits_{(x_i,n_i)}\exp(-n_i(s-\epsilon(\delta)-\gamma)+S_{n_i}\varphi(x_i)),
\end{align*}
where $\gamma$ is a function such that $\exp(2g(n_i,L(\mathcal U))||\varphi||)\leq \exp(\gamma n_i) $ as $n_i\geq N$ and $\gamma\to0$ as diam $\mathcal U\to0$.
Thus taking $N\to\infty,$ $P_Z'(\varphi,\mathcal U,g)\leq P_Z(\varphi,\delta,g)-\epsilon(\delta)-\gamma$. As $\delta\to0,$ we have
\begin{align*}
P'_Z(\varphi,g)\leq P_Z(\varphi,g).
\end{align*}
To sum up, we have
\begin{align*}
 P_Z'(\varphi,g)= P_Z(\varphi,g).
 \end{align*}
\end{proof}
Now we start to compare the new topological pressure to the classical Pesin pressure.
\begin{theorem}\label{thm2}
Let $(X,f)$ be a TDS, $\varphi$ be a continuous function, $g$ be a mistake function. Then
\begin{align*}
P_Z(\varphi,g)=P_Z(\varphi).
\end{align*}
\end{theorem}

\begin{proof}
First, we show that $P_Z(\varphi,g)\leq P_Z(\varphi).$
For any ${\bf U}\in\mathcal S(\mathcal U),$ we have
\begin{align*}
Z({\bf U})\subset Z({\bf U}^g).
\end{align*}
For any $\epsilon^*>0$, there is $\mathcal G^*\in \mathcal S(Z,\mathcal U,N)$  such that
\begin{align*}
m'(Z,\varphi,\mathcal U,s, N)&=\inf\limits_{\mathcal S(Z,\mathcal U,N)}\left\{\sum\limits_{{\bf U}\in\mathcal G}\exp(-sm({\bf U})+\sup\limits_{x\in Z({\bf U})}S_{m({\bf U})}\varphi(x))\right\}\\
&\geq \sum\limits_{{\bf U}\in\mathcal G^*}\exp(-(s+\epsilon^*)m({\bf U})+\sup\limits_{x\in Z({\bf U})}S_{m({\bf U})}\varphi(x)) \\
&\geq \sum\limits_{{\bf U}^g\in\mathcal G'}\exp(-(s+\epsilon^*)m({\bf U})+\sup\limits_{x\in Z({\bf U}^g)}S_{m({\bf U})}\varphi(x)-2g(m({\bf U}),L(\mathcal U))||\varphi||) \\
&\geq m'(Z,\varphi,\mathcal U,s+\epsilon^*+\gamma, N,g),
\end{align*}
where $\mathcal G'=\{{\bf U}^g:{\bf U}\in\mathcal G^*\},\gamma$ is a function such that
$2g(m({\bf U}),L(\mathcal U))||\varphi||\leq m({\bf U})\gamma$ for all $m({\bf U})\geq N$ and $\gamma\to0$ as  diam $\mathcal U\to0$.
Taking $N\to\infty,$ it yields
\begin{align*}
P'_Z(\varphi,\mathcal U)\geq P'_Z(\varphi,\mathcal U,g)+\epsilon^*+\gamma.
\end{align*}
As $\epsilon^*\to0$ and diam $\mathcal U\to0,$ we have
\begin{align*}
P'_Z(\varphi,g)\leq P'_Z(\varphi).
\end{align*}
Next, we show that $P_Z(\varphi,g)\geq P_Z(\varphi).$ For any ${\bf U}^g\in\mathcal S_m^g(\mathcal U),$ it can be  composed by
$\sum\limits_{i=0}^{g(m,L(\mathcal U))} \binom{m}{i}\sharp (\mathcal U)^{i-1}$ strings ${\bf U}\in\mathcal S_m(\mathcal U)$ such that
$Z({\bf U}^g)\subset \bigcup\limits_{{\bf U}\in\mathcal H} Z({\bf U})$, where $\mathcal H$ is the collection of such ${\bf U}.$
By Stirling formula, there exists $\gamma>0$ such that $\sum\limits_{i=0}^{g(m,L(\mathcal U))} \binom{m}{i}\sharp (\mathcal U)^i\leq \exp(m\gamma)$ for all $m\geq N$
and $\gamma\to0$  as diam $\mathcal U\to0$. For any $\epsilon^*>0,$ there is $\mathcal G''\in\mathcal S^g(Z,\mathcal U,N)$ such that
\begin{align*}
m'(Z,\varphi,\mathcal U,s, N, g)&=\inf\limits_{\mathcal S^g(Z,\mathcal U,N)}\sum\limits_{{\bf U}^g\in\mathcal G^g}\exp(-sm({\bf U})+\sup\limits_{x\in Z({\bf U}^g)}S_{m({\bf U})}\varphi(x))\\
&\geq \sum\limits_{{\bf U}^g_i\in\mathcal G''} \exp(-(s+\epsilon^*)m({\bf U}_i)+\sup\limits_{x\in Z({\bf U}_i^g)}S_{m({\bf U}_i)}\varphi(x))\\
&\geq \sum\limits_{{\bf U}_i\in\mathcal G'''}\exp(-n_i(s+\epsilon^*+\gamma)+\sup\limits_{x\in Z({\bf U}_i)}S_{m({\bf U}_i)}\varphi(x))\\
&\geq m'(Z,\varphi,\mathcal U,s+\epsilon^*+\gamma, N),
\end{align*}
where $\mathcal G'''=\{{\bf U}: {\bf U}_i^g\in \mathcal G''\}$ is the collection of the elements of all $\mathcal H$ corresponding to each ${\bf U}^g\in\mathcal G''.$ Taking $N\to\infty,$ it yields
\begin{align*}
P'_Z(\varphi,\mathcal U,g)\geq P'_Z(\varphi,\mathcal U)+\epsilon^*+\gamma.
\end{align*}
Letting $\epsilon^*\to0$ and diam $\mathcal U\to0,$ we have
\begin{align*}
P'_Z(\varphi,g)\geq P'_Z(\varphi).
\end{align*}
This completes the proof.
\end{proof}
As an application, we study the pressure defined by average metric at the end of the paper.
Let $(X,f)$ be a TDS. The average metric on $X$ is given by
\begin{align*}
\overline{d_{n}}(x,y):=\dfrac{1}{n}\sum_{i=0}^{n-1}d(T^{i}x,T^{i}y).
\end{align*}
The average dynamical ball with center  $x$ and radius $\epsilon$ is denoted by
$
B_{\overline{d}_{n}}(x,\epsilon).
$
 In \cite{grogerjager}, Groger and Jager proved the classical entropy of the whole system is the same as the entropy defined by replace Bowen balls with  average dynamical balls. We will use Theorem \ref{thm2} to show the classical Pesin pressure on any subset is equal to the pressure defined by replace Bowen balls with  average dynamical balls.

In fact, we choose $g(n,\epsilon)=n\epsilon.$ By the following lemma, it is easy to obtain the desired result and we omit the proof.
\begin{lemma}
For any $x\in X$, $n\in\mathbb{N}$ and $\epsilon>0$, we have $$B_{n}(x,\epsilon)\subset B_{\overline{d}_{n}}(x,\epsilon)\subset B_{n}(g;x,\sqrt{\epsilon}).$$
\begin{proof}
For any $y\in X$, if $d_{n}(x,y)<\epsilon$, then $\overline{d}_{n}(x,y)<\epsilon$, so we have $B_{n}(x,\epsilon)\subset B_{\overline{d_{n}}}(x,\epsilon)$.
Set
\begin{align*}
 I_{1}&=\{i\in\Lambda_{n}:d(T^{i}x,T^{i}y)<\sqrt{\epsilon}\}, \\
 I_{2}&=\{i\in\Lambda_{n}:d(T^{i}x,T^{i}y)\geq\sqrt{\epsilon}\},
\end{align*}
where $\Lambda_{n}=\{0,1,\cdots,n-1\}.$
Since
\begin{align*}
 \dfrac{1}{n}\sum_{i=0}^{n-1}d(T^{i}x,T^{i}y)&=\frac{\sum_{i\in I_{1}}d(T^{i}x,T^{i}y)+\sum_{i\in I_{2}}d(T^{i}x,T^{i}y)}{n}\\
 &\geq\frac{\sum_{i\in I_{1}}d(T^{i}x,T^{i}y)+\sqrt{\epsilon}\sharp I_{2}}{n}\geq\dfrac{\sqrt{\epsilon}}{n}\sharp I_{2},
\end{align*}
then if $\overline{d}_{n}(x,y)<\epsilon$, we have $$\sharp I_{2}\leq n\sqrt{\epsilon}.$$
Therefore, we have $B_{\overline{d_{n}}}(x,\epsilon)\subset B_{n}(g;x,\sqrt{\epsilon})$.
\end{proof}
\end{lemma}

\noindent\textbf{Acknowledgments:}

The first author was supported by NNSF of China (11601235),  NSF of Jiangsu Province (BK20161014), NSF of the Jiangsu Higher Education Institutions of China (16KJD110003),  China Postdoctoral Science Foundation (2016M591873) and China Postdoctoral Science Special Foundation (2017T100384). The third author was  supported by NNSF of China (11671208, 11431012). The work was also funded by the Priority Academic Program Development of Jiangsu Higher Education Institutions.

\end{document}